\documentclass[12pt,reqno]{amsart}
\usepackage{cmap,mathtools}
\usepackage[T1]{fontenc}
\usepackage[utf8]{inputenc}
\usepackage[english]{babel}
\usepackage{amsfonts,amssymb,amsmath}
\usepackage{footnote,bigints}
\usepackage{array}
\usepackage{adjustbox}

\usepackage{graphicx}
\usepackage{caption}
\usepackage{subcaption}
 \usepackage{enumitem}

\allowdisplaybreaks
\usepackage[pagewise, mathlines]{lineno}
\usepackage{graphicx}
\usepackage{epstopdf}
\usepackage{a4wide}
\usepackage{xcolor}
\usepackage[colorlinks=true,linktocpage,pdfpagelabels,
bookmarksnumbered,bookmarksopen]{hyperref}
\definecolor{ForestGreen}{rgb}{0.1,0.6,0.05}
\definecolor{EgyptBlue}{rgb}{0.063,0.1,0.6}
\hypersetup{
	colorlinks=true,
	linkcolor=EgyptBlue,         
	citecolor=ForestGreen,
	urlcolor=olive
}

\usepackage[hyperpageref]{backref}

\newtheorem{theorem}{Theorem}[section]
\newtheorem{proposition}{Proposition}[section]
\newtheorem{definition}{Definition}[section]
\newtheorem{lemma}{Lemma}[section]
\newtheorem{remark}{Remark}[section]
\newtheorem{cor}{Corollary}[section]

\newtheorem{conjecture}{Conjecture}[section]

\numberwithin{equation}{section}
\numberwithin{theorem}{section}
\numberwithin{equation}{section}
\numberwithin{theorem}{section}

\usepackage{ulem}
\usepackage{xcolor}
\usepackage[colorlinks=true,linktocpage,pdfpagelabels,
bookmarksnumbered,bookmarksopen]{hyperref}
\definecolor{ForestGreen}{rgb}{0.1,0.6,0.05}
\definecolor{EgyptBlue}{rgb}{0.063,0.1,0.6}
\hypersetup{
	colorlinks=true,
	linkcolor=EgyptBlue,         
	citecolor=ForestGreen,
	urlcolor=olive
}

\usepackage[hyperpageref]{backref}
\usepackage[foot]{amsaddr}

\DeclareUnicodeCharacter{2212}{-}

\title [Steklov eigenvalues on Symmetric domains]{Bounds for higher Steklov and mixed Steklov Neumann eigenvalues on domains with holes}

\author{Sagar Basak$^1$ \and Sheela Verma$^*$}
\address{$1$  Department of Mathematical Sciences, Indian Institute of Technology (BHU), Varanasi, India}
\email{sagarbasak.rs21@iitbhu.ac.in}

\address{$*$  Corresponding author, Department of Mathematical Sciences, Indian Institute of Technology (BHU), Varanasi, India}
\email{sheela.mat@iitbhu.ac.in}

\DeclareUnicodeCharacter{2212}{-}

\begin{document}

\begin{abstract}
In this article, we study Steklov eigenvalues and mixed Steklov Neumann eigenvalues on a smooth bounded domain in $\mathbb{R}^{n}$, $n  \geq 2$, having a spherical hole. We focus on two main results related to Steklov eigenvalues. First, we obtain explicit expression for the second nonzero Steklov eigenvalue on concentric annular domain. Secondly, we derive a sharp upper bound of the first $n$ nonzero Steklov eigenvalues on a domain $\Omega \subset \mathbb{R}^{n}$ having symmetry of order $4$ and a ball removed from its center. This bound is given in terms of the corresponding Steklov eigenvalues on a concentric annular domain of the same volume as $\Omega$. Next, we consider the mixed Steklov Neumann eigenvalue problem on $4^{\text{th}}$ order symmetric domains in $\mathbb{R}^{n}$ having a spherical hole and obtain upper bound of the first $n$ nonzero eigenvalues. We also provide some examples to illustrate that symmetry assumption in our results is crucial. Finally, We make some numerical observations about these eigenvalues using FreeFEM++ and state them as conjectures.
\end{abstract}

\maketitle

\section{introduction} 
Among all domains with certain geometric constraints, finding a domain which optimizes eigenvalues of a differential operator is a classical problem in Spectral Geometry and it is referred as shape optimization problem. In this work, we are interested in shape optimization problem for eigenvalues of the Dirichlet to Neumann operator
\begin{align*}
    \mathcal{D} : L^{2}(\mathcal{C}_1) & \longrightarrow  L^{2}(\mathcal{C}_1), \\
    f & \longmapsto \frac{\partial \tilde{f}}{\partial \nu}.
\end{align*}
Here $\Omega$ is a bounded connected domain having Lipschitz boundary $\partial \Omega = \mathcal{C}_1 \sqcup \mathcal{C}_2$ and $\nu$ is unit outward normal to $\partial \Omega$. The function $\tilde{f}$ represents the Harmonic extension of $f$ to $\Omega$ satisfying certain conditions on $\mathcal{C}_2$.

If $\mathcal{C}_2 = \emptyset$ and function $\tilde{f}$ is the Harmonic extension of $f$ to $\Omega$, then eigenvalues of operator $\mathcal{D}$ are same as eigenvalues of the following Steklov problem
\begin{align} \label{eqn: SP}
    \begin{cases}
       \Delta u =0 \quad in \,\, \Omega, \\
        \frac{\partial u}{\partial \nu} = \sigma u \quad on \,\, \partial \Omega (= \mathcal{C}_1).
    \end{cases}
\end{align}
 Eigenvalues of Problem \eqref{eqn: SP} are positive and form an increasing sequence $0 = \sigma_{0}(\Omega) < \sigma_{1}(\Omega) \leq \sigma_{2}(\Omega) \leq \cdots \nearrow \infty.$

If function $\tilde{f}$ is the Harmonic extension of $f$ to $\Omega$ satisfying Neumann boundary condition on $\mathcal{C}_2$, then eigenvalues of operator $\mathcal{D}$ are same as eigenvalues of the following mixed Steklov Neumann problem
\begin{align} \label{eqn: SNP}
    \begin{cases}
       \Delta u =0 \quad in \,\, \Omega, \\
        \frac{\partial u}{\partial \nu} = \mu u \quad on \,\, \mathcal{C}_1 \\
        \frac{\partial u}{\partial \nu} = 0 \quad on \,\, \mathcal{C}_2.
    \end{cases}
\end{align}
Eigenvalues of Problem \eqref{eqn: SNP} are positive and form an increasing sequence $0 = \mu_{0}(\Omega) < \mu_{1}(\Omega) \leq \mu_{2}(\Omega) \leq \cdots \nearrow \infty.$

In literature, various sharp bounds have been proved for Steklov eigenvalues but very few results are available related to the mixed Steklov Neumann eigenvalues. Now, We state some results related to our study on Steklov and mixed Steklov Neumann eigenvalues.

Weinstock \cite{weinstock1954inequalities} obtained first isoperimetric sharp bound for the first nonzero Steklov eigenvalue and proved that among all simply connected planar domains of fixed perimeter, disk maximizes the first nonzero Steklov eigenvalue. This result can not be generalised to higher dimensions for around 60 years. Recently in \cite{bucur2021weinstock}, this result is generalized to convex domains contained in $\mathbb{R}^{n}, n \geq 2$ by Bucur et. al. The authors proved that among all bounded convex set $\Omega \subset \mathbb{R}^{n}$ of fixed perimeter $L$, ball maximizes the first nonzero Steklov eigenvalue i.e.,
\begin{align} \label{Bucur_convex}
    \sigma_1 (\Omega) \leq \sigma_1 (B),
\end{align}
where $B$ is a ball in $\mathbb{R}^{n}$ with perimeter $L$.

For general simply connected domains in higher dimensions, Fraser and Schoen \cite{fraser2019shape} proved that this result is not true by showing the existence of a smooth contractible domain having same perimeter as of the unit ball but having larger first nonzero Steklov eigenvalue. A quantitative version of inequality \eqref{Bucur_convex} has been proved in \cite{gavitone2020quantitative}. The theory of extremal metrics on a surface have been developed for Steklov eigenvalues in \cite{fraser2011first, fraser2013minimal, fraser2016sharp}. For stability results related to first nonzero Steklov eigenvalue on Planar domains, see \cite{bucur2021stability, girouard2017spectral}.

Note that all the above isoperimetric bounds are proved for bounded smooth domains under perimeter constraints. Several other estimates have also been obtained for Steklov eigenvalues on bounded smooth domains under volume constraints. In \cite{brock2001isoperimetric}, F. Brock proved that among all bounded domains in $\mathbb{R}^{n}$ with Lipschitz boundary and fix volume, ball minimizes harmonic mean of the first $n$- nonzero Steklov eigenvalues. This result has been generalised for the harmonic mean of the first $(n-1)$ nonzero Steklov eigenvalues on bounded domains contained in Hyperbolic space \cite{verma2021isoperimetric} and in curved spaces \cite{chen2024upper}. Similar result has also been proved for the first nonzero Steklov eigenvalue on bounded domains contained in noncompact rank-$1$ symmetric spaces in \cite{BinoySanthanam2024}.


The problem of finding bounds for the first Laplace eigenvalue on doubly connected domains with different boundary conditions is an interesting problem and many results for planar domains have been proved in this direction in $19^{\text{th}}$ century. To the best of our knowledge, first such result was due to Makai \cite{makai1959principal}, where upper bound for the first Dirichlet eigenvalue on ring shaped planar domain was obtained. For the mixed Dirichlet Neumann problem, first such result was proved by Payne and Weinberger \cite{payne1961some}, where authors consider a bounded domain having several holes (multi connected domain) with Dirichlet condition on the outer boundary and Neumann condition on other boundaries. Among such domains of given area and given perimeter of outer boundary, it was proved in \cite{payne1961some} that first Laplace eigenvalue is maximum for concentric annular domain. In the last decades, this problem has been studied by many researchers for various boundary conditions like the Mixed Dirichlet Neumann problem \cite{anoop2020reverse, anoop2021shape}, the Mixed Steklov Dirichlet problem \cite{basak2023sharp, gavitone2023isoperimetric,verma2020eigenvalue}. For results related to Dirichlet boundary condition, see \cite{anisa2005two, chorwadwala2013two, el2008extremal}. In \cite{ftouhi2022place}, Steklov eigenvalue problem on an annular domain in $\mathbb{R}^{n}$ i.e., a ball in $\mathbb{R}^{n}$ having circular obstacle, was considered. The author proved that among all position of inner ball, the first nonzero Steklov eigenvalue is maximized if the outer and inner ball have the same center i.e., the first nonzero Steklov eigenvalue is maximum on the concentric annular domain.

Sharp bounds for higher eigenvalues of Laplacian with various boundary conditions have also been obtained on doubly connected domains under certain symmetry restrictions. In \cite{anoop2022szego1}, authors proved that among all multi-connected domains in $\mathbb{R}^{n}$ satisfying certain symmetry conditions, concentric annular domain maximizes the first nonzero Neumann eigenvalue. This result was later extended to domain contained in simply connected space forms, see \cite{anoop2022szego}. Recently similar bounds have been proved for Robin-Neumann boundary condition \cite{anoop2023reverse} and for the mixed Steklov Dirichlet condition \cite{basak2023sharp}. 

Now we present results related to mixed Steklov Neumann eigenvalues. Certain inequalities between the mixed Steklov-Dirichlet and the Steklov-Neumann eigenvalues of the Laplacian have been obtained in \cite{banuelos2010eigenvalue}. In \cite{hassannezhad2020eigenvalue}, bounds on the Riesz means of the mixed Steklov-Neumann and Steklov-Dirichlet eigenvalue problem have been studied on a bounded domain in $\mathbb{R}^{n}$. Considering link between mixed Steklov Neumann eigenvalues and mixed Steklov Dirichlet eigenvalues, sharp isoperimetric bound for these eigenvalues is obtained in \cite{arias2024applications}. Authors also provided full asymptotics for these mixed Steklov problems on arbitrary surfaces. To best of our knowledge, higher Steklov eigenvalues and mixed Steklov Neumann eigenvalues are not explored much. With this motivation, in this article, we are interested in isoperimetric bounds of higher Steklov eigenvalues and higher Steklov Neumann eigenvalues on doubly connected domains.

The rest part of this article is organized as follows: In Section \ref{sec:annular domain}, we study eigenvalues and eigenfunctions of Steklov eigenvalue problem on concentric annular domain and find explicit expression for the second nonzero Steklov eigenvalue counted without multiplicity (Theorem \ref{thm:increasing}). Some auxiliary results about the behavior of Steklov eigenfunctions on annular domain have been proved in Section \ref{sec:int. ineq.}. Then in Section \ref{sec: iso. bound}, we prove that concentric annular domain is the optimal domain for the first $n$ nonzero Steklov eigenvalues among all multi connected domains in $\mathbb{R}^{n}$ having certain symmetry restrictions (Theorem \ref{thm:isoperimetric}). The mixed Steklov Neumann problem have been studied in Section \ref{sec:SN problem} and similar bounds as in Theorem \ref{thm:isoperimetric} have been obtained for the mixed Steklov Neumann eigenvalues (Theorem \ref{thm:isoperimetricSN}). In Section \ref{Counter examples}, we provide examples of some domains to show that the symmetry condition assumed in Theorem \ref{thm:isoperimetric} and \ref{thm:isoperimetricSN} can not be dropped. In last Section \ref{sec:numerical observation}, we observed interesting monotonicity behaviour of the first two nonzero Steklov eigenvalues and mixed Steklov Neumann eigenvalues on certain domains. For example, on any annular domain (ball with a spherical hole) in $\mathbb{R}^{2}$, the first mixed Steklov Neumann eigenvalue have multiplicity 2. Further, it decreases as the distance between centers of both ball increases.

\section{Steklov eigenvalue problem on concentric annular domain} 
\label{sec:annular domain}
Let $B_L$ and $B_1$ be balls in $\mathbb{R}^n$ of radius $L (> 1)$ and $1$, respectively centered at the origin. Consider the Steklov eigenvalue problem on annular domain $\Omega_{0} = B_L \setminus \bar{B_1}$,
\begin{align}
    \begin{cases}
        \Delta u =0 \quad in \,\, \Omega_0, \\
        \frac{\partial u}{\partial \nu} = \sigma u \quad on \,\, \partial \Omega_0.
    \end{cases} \label{stk on ball-ball}
\end{align}
In this section, we first calculate all eigenvalues and eigenfunctions of \eqref{stk on ball-ball} and then among them we find the second nonzero Steklov eigenvalue on annular domain counted without multiplicity (Theorem \ref{thm:increasing}). 

\subsection{Preliminaries}
Due to symmetry of annular domain $\Omega_{0}$, any eigenfunction of \eqref{stk on ball-ball} will be of the form $u(r,\omega)=f(r)g(\omega)$, where $f(r)$ is a radial function defined on $[1, L]$ and $g(\omega)$ is an eigenfunction of $\Delta_{S^{n-1}}$. Recall that the spectrum of $\Delta_{S^{n-1}}$ is $\{l(l+n-2), l \in \mathbb{N} \cup \{0\}\}$. The multiplicity of eigenvalue $l(l+n-2)$ is same as the dimension of the set of all homogeneous harmonic polynomials $H_l$ in $n$ variables and of degree $l \in \mathbb{N} \cup \{0\}$. It is well known that 
\begin{align*}
    \text{dim } H_{l} = \frac{2l+n-2}{l+n-2} {l+n-2 \choose {l}}.
\end{align*}
Eigenfunctions corresponding to eigenvalue $l(l+n-2)$ are precisely the set of all spherical harmonics of degree $l$.  The set of spherical harmonics are defined as the restriction to $\mathbb{S}^{n-1}$  of homogeneous harmonic polynomials in $n$ variables. For more details about the spectrum and eigenfunctions of $\Delta_{S^{n-1}}$, see \cite{anoop2022szego}.
The Laplacian of $u(r,\omega)$ takes the form,
\begin{align*}
     \Delta u(r, \omega) &= g(\omega)\left(-f''(r)- \frac{n-1}{r} f'(r) \right) + \frac{f(r)}{r^2} \Delta_{S ^{n-1}} g(\omega)\\
     &= g(\omega) \left(-f''(r)- \frac{n-1}{r} f'(r) + \frac{f(r)}{r^2} l(l+n-2) \right).
\end{align*}
Since $u(r,\omega)$ is a solution of (\ref{stk on ball-ball}), function $f(r)$ satisfies 
\begin{eqnarray} 
    \begin{array}{ll} \label{ode}
        -f''(r)-\frac{n-1}{r}f'(r)+ \frac{l(l+n-2)}{r^2}f(r)=0 ~\mbox{ for } ~ r \in (1,L), \\
    f'(1)=-\sigma f(1), \,\, f'(L)=\sigma f(L).
    \end{array}
\end{eqnarray}
Solving this ordinary differential equation for function $f(r)$, we get 
\begin{equation*}
    f(r)= \begin{cases}
         C_1 \ln r + C_2  &\quad l=0, n=2, \\
         C_1 r^{l}+C_2 \frac{1}{r^{l+n-2}}  &\quad \text{ otherwise},
    \end{cases}
\end{equation*}
for some constants $C_1$ and $C_2$. To calculate values of $C_1$ and $C_2$, we use the boundary condition of (\ref{ode}) and get a system of two linear equations. 

\textbf{Case I.} For $l=0, n=2$ 
\begin{equation} \label{eqn:linear system}
   \begin{cases}
       C_1 = -C_2 \sigma, \\
      C_1\frac{1}{L}=\sigma (C_1ln(L)+C_2).
   \end{cases} 
\end{equation}
This system has non trivial solution if $\sigma^2 L\,\, \ln L-\sigma (L+1)=0$. This equation have two solutions say, $\sigma_{0, 1} = 0$ and $\sigma_{0, 2} = \frac{1+L}{L \, \ln L}$. Thus $\sigma_{0, 1}$ and $\sigma_{0, 2}$ are eigenvalues of \eqref{ode} corresponding to eigenfunctions  $f_{0, 1}$ and $f_{0, 2}$, respectively, where $ f_{0,i}(r) = 1-\sigma_{0, i} \,ln(r), i = 1,2$ and $r \in [1,L]$.

\textbf{Case II.} For the remaining values of $l$ and $n$, 
\begin{equation}
    \begin{cases}
        C_1(l+\sigma)-C_2(l+n-2-\sigma)=0,\\
        C_1(lL^{l-1}-L^l\sigma)-C_2(\frac{l+n-2}{L^{l+n-1}}+\frac{1}{L^{l+n-2}}\sigma)=0.
    \end{cases}
\end{equation}
This system has nontrivial solution if $ \tilde{A} \sigma^2 + \tilde{B} \sigma + \tilde{C} = 0$, Where 
\begin{align*}
     \begin{cases}
         \tilde{A}=L(L^{2l+n-2}-1) \\
         \tilde{B}=-\{ lL^{2l+n-2}+L^{2l+n-1}(l+n-2)+lL+l+n-2 \} \\
         \tilde{C} = l(l+n-2)(L^{2l+n-2}-1).
     \end{cases}
 \end{align*}
We compute the discriminant $\mathcal{D}$ of the quadratic equation and verify that $\mathcal{D} > 0$. Note that 
 \begin{align*}
    \mathcal{D}&=\tilde{B}^2 -4\tilde{A} \tilde{C}\\
              &= \{(l+n-2)L^{2l+n-1}+lL^{2l+n-2}+lL+l+n-2\}^2-4(l+n-2)lL(L^{2l+n-2}-1)^2 \\ 
              & \geq \{(l+n-2)L^{2l+n-1}+lL^{2l+n-2}\}^2-4(l+n-2)lL(L^{2l+n-2}-1)^2 \\
              &\geq (L^{2l+n-2}-1)^2\{(l+n-2)L+l\}^2-4(l+n-2)lL(L^{2l+n-2}-1)^2 \\
              & =(L^{2l+n-2}-1)^2 \{(l+n-2)L-l\}^2 > 0.
\end{align*}
Thus equation $ \tilde{A} \sigma^2 + \tilde{B} \sigma + \tilde{C} = 0$ have two different real solutions, say 
\begin{align*}
\sigma_{l, 1} = \frac{-\tilde{B}-\sqrt{\mathcal{D}}}{2\tilde{A}} < \frac{-\tilde{B}+\sqrt{\mathcal{D}}}{2\tilde{A}} = \sigma_{l, 2}.
\end{align*}
The eigenfunction of \eqref{ode} corresponding to $\sigma_{l, i}, i=1, 2$ is 
\begin{align} \label{eqn: eigenfunction ODE}
    f_{l,i}(r)= r^l + \frac{(l+\sigma_{l,i})}{l+n-2-\sigma_{l,i}} \frac{1}{r^{l+n-2}}, \quad r \in [1,L].
\end{align}

\begin{remark} \label{rmk:explicit expression}
Note that eigenfunctions of \eqref{stk on ball-ball} are of the form $f_{l,i}(r) g(\omega)$ corresponding to the eigenvalues $\sigma_{l,i}$,\,\, $l\geq 0,$ $i= 1,2$, and $g(\omega)$ is an eigenfunction of $\Delta_{\mathbb{S}^{n-1}}$ corresponding to eigenvalue $l(l+n-2)$. The following explicit expressions for $\sigma_{l,i}$ will be used in the next subsection to prove Theorem \ref{thm:increasing}:

\begin{enumerate}
\item $\sigma_{0, 1} = 0.$ 
\item Value of $ \sigma_{0, 2}$ depends on the value of $n$,
\begin{align*}
    \sigma_{0, 2} = 
\begin{cases}
\frac{1+L}{L \, \ln L},  & \quad \ n=2 \\
\frac{(n-2)(1+L^{n-1})}{(L^{n-1}-L)}, & \quad \ n>2.\\
\end{cases}
\end{align*}
\item $\sigma_{1,2} = \frac{(L+n-1)+L^n(1+L(n-1))+\sqrt{\left((L+n-1)+L^n(1+L(n-1))\right)^2-4L(L^n-1)^2(n-1)}}{2L(L^{n}-1)}.$
\item $\sigma_{2,1} = \frac{L^{n+2}(2+Ln)+(n+2L)-\sqrt{\left( L^{n+2}(2+Ln)+(n+2L) \right)^2-8Ln(L^{n+2}-1)^2}}{2L(L^{n+2}-1)}.$
\end{enumerate}
\end{remark}

\subsection{Second nonzero Steklov eigenvalue counted without multiplicity}

It was proved in \cite{ftouhi2022place} that the first nonzero Steklov eigenvalue on annular domain $\sigma_{1}(\Omega_0) = \sigma_{1,1}$ is of multiplicity $n$ i.e., $\sigma_{1}(\Omega_0) = \sigma_{2}(\Omega_0)= \cdots = \sigma_{n}(\Omega_0) = \sigma_{1,1}$. In this subsection, we find the next smallest Steklov eigenvalue $\sigma_{n+1}(\Omega_0)$ on annular domain and provide it's explicit expression. Precisely, we prove the following theorem:
\begin{theorem} \label{thm:increasing}
The second nonzero Steklov eigenvalue, counted without multiplicity, on concentric annular domain $\Omega_0$ is
\begin{align*}
\sigma_{n+1}(\Omega_0) = \sigma_{2,1} = \frac{L^{n+2}(2+Ln)+(n+2L)-\sqrt{\left( L^{n+2}(2+Ln)+(n+2L) \right)^2-8Ln(L^{n+2}-1)^2}}{2L(L^{n+2}-1)} 
\end{align*}
with multiplicity $\frac{(n+2)(n-1)}{2}.$
\end{theorem}

To prove Theorem \ref{thm:increasing}, we need the following lemmas.

    \begin{lemma} \label{lem: eigenvalue ineq. for planar}
        Let $L, \sigma_{0,2}, \sigma_{1,2}$ and $\sigma_{2,1}$ be defined as the above. Then, for $n=2$,
    \begin{enumerate}
        \item $ \sigma_{0,2} \leq \sigma_{1,2}.$
        \item $ \sigma_{2,1} \leq \sigma_{0,2}.$
    \end{enumerate}
            
    \end{lemma}
    \begin{proof} Note that for $n=2$,
    \begin{align*}
    \sigma_{0,2}= \frac{L+1}{L\, lnL}, \quad \sigma_{1,2} = \frac{(L^2+1)+\sqrt{(L-1)^4+4L^2}}{2L(L-1)} \\
    \sigma_{2,1}=\frac{(1+L^4)-\sqrt{(1+L^4)^2-4L(1+L^2)^2(1-L)^2}}{L(L-1)(1+L^2)}.
    \end{align*}
    \begin{enumerate}
        \item  Inequality $\sigma_{0,2} \leq \sigma_{1,2}$ is equivalent to prove that 
        \begin{align*}
            2(L^2-1)\leq \left((L^2+1)+\sqrt{(L-1)^4+4L^2}\right)lnL.
        \end{align*}
        Again this inequality is true if 
        \begin{align*}
            2(L^2-1)\leq (L+1)^2 \ln L.
        \end{align*}
        We take $\tilde{h}(t)=(t+1)^2 \ln t - 2(t^2-1), 1 \leq t \leq L$. Note that $\tilde{h}(1)= 0, \tilde{h}'(1)= 0, \tilde{h}''(1)= 0$ and $\tilde{h}'''(t) \geq 0$ for all $1 \leq t \leq L$. This gives $\tilde{h}''(t)$ is an increasing function and $\tilde{h}''(t) \geq \tilde{h}''(1) = 0$ for all $1 \leq t \leq L$. Using this argument repeatedly, we get $\tilde{h}(t) \geq 0$ for all $1 \leq t \leq L$. In particular, $\tilde{h}(L) \geq 0,$ which gives $\sigma_{0,2} \leq \sigma_{1,2}$.

        \item Inequality $\sigma_{2,1} \leq \sigma_{0,2}$ is equivalent to prove that
        \begin{align} \label{inequality}
            \left((1+L^4)-\sqrt{(1+L^4)^2-4L(1+L^2)^2(1-L)^2} \right) \ln L \leq (L^4-1).
        \end{align}
        A simple calculation shows that
        \begin{align*}
         (1+L^4)-\sqrt{(1+L^4)^2-4L(1+L^2)^2(1-L)^2} \leq \frac{2(1+L^4)}{(L+1)}.  
        \end{align*}
        So the inequality (\ref{inequality}) is true if 
        \begin{align*}
            \ln L \leq \frac{(L^4-1)(L+1)}{2(1+L^4)}.
        \end{align*}
        We introduce $w(t)=\frac{(t^4-1)(t+1)}{2(1+t^4)}- \ln t, 1 \leq t \leq L$. The idea is to show that $w(t)$ is an increasing function of $t$ by proving that $w'(t) \geq 0$. This combining with $w(1) = 0$ will give the desired result. We have
        \begin{align*}
          w'(t)=\frac{2t^9-4t^8+16t^5+8t^4-2t-4}{4t(1+t^4)}.   
        \end{align*}
        Consider the function $\tilde{w}(t)=2t^9-4t^8+16t^5+8t^4-2t-4.$ It is easy to check that the $k^{\text{th}}$ derivative of $\tilde{w}(t)$ at point $t=1$ i.e., $\tilde{w}^{(k)}(1)\geq 0 $ for $k=1,2,3$, and $\tilde{w}^{(4)}(t)\geq 0$ for all $1 \leq t \leq L$. Now using the same arguments as in the first part of this Lemma, we can prove that $\tilde{w}(t) \geq 0$. Since the denominator term in $w'(t)$ is always positive, we get that $w'(t) \geq 0$.
    \end{enumerate}
        
    \end{proof}

\begin{lemma} \label{positive lemma}
    For $n\geq 3$, consider
    \begin{align*}
       h(t) = (n-2)t^{2n+1}-(n+2)t^{2n}+(n-2)t^{n+3}+(3n-2)t^{n+2}-2n t^{n-1}+4t-2(n-2),
    \end{align*}
    then $h(L) \geq 0$ for each $L > 1$.
\end{lemma}

\begin{proof}
    For a fix $L>1$, let $h^{(k)}(t)$ denotes the $k$th derivative of function $h$ at a point $t \in [1,L]$. We first prove that $h^{(k)}(1) \geq 0$ for all $0 \leq k \leq 7$ and $h^{(8)}(t) \geq 0$. Then using the same idea as in Lemma \ref{lem: eigenvalue ineq. for planar}, our conclusion follows. We have
\begin{align*}
   h(1) =  0, \quad h^{(1)}(1) = 2n^2-8 \geq 0, \quad h^{(2)}(1) = 2(n-2)(n+2)^2 \geq 0.
\end{align*}
For $3 \leq k \leq 2n+1,$
\begin{align*}
    h^{(k)}(t)= & (n-2)\frac{(2n+1)!}{(2n+1-k)!}t^{(2n+1-k)}-(n+2)\frac{(2n)!}{(2n-k)!}t^{(2n-k)}+(n-2)\frac{(n+3)!}{(n+3-k)!}t^{(n+3-k)}\\
    &+(3n-2)\frac{(n+2)!}{(n+2-k)!}t^{(n+2-k)}-2n\frac{(n-1)!}{(n-1-k)!}t^{(n-1-k)},
\end{align*} 
with the condition that the power of $t$ in each term is non-negative; otherwise, we take that term to be zero. By substituting the values of $k$, we can check for $ k=3,4,\dots 7, \,\,h^{(k)}(1) \geq 0 $. For $k=8$,
\begin{align} \nonumber
     h^{(8)}(t)=&(n-2)\frac{(2n+1)!}{(2n+1-8)!} t^{(2n+1-8)}-(n+2)\frac{(2n)!}{(2n-8)!}t^{(2n-8)} \\ \nonumber
     & \quad +(n-2)\frac{(n+3)!}{(n+3-8)!}t^{(n+3-8)} +(3n-2)\frac{(n+2)!}{(n+2-8)!}t^{(n+2-8)}\\ \label{8th derivative}
    & \qquad -2n\frac{(n-1)!}{(n-1-8)!}t^{(n-1-8)}.
\end{align}

Note that  $\bigg((n-2)\frac{(2n+1)!}{(2n+1-8)!} - (n+2)\frac{(2n)!}{(2n-8)!}\bigg)$ and
$\bigg((n-2)\frac{(n+3)!}{(n+3-8)!}+(3n-2)\frac{(n+2)!}{(n+2-8)!}-2n\frac{(n-1)!}{(n-1-8)!}\bigg)$ are positive, we get that $h^{(8)}(t) \geq 0$.
\end{proof}
  
\begin{lemma} \label{lem: inequality1for higher order}
For $n \geq 3$, $L>1$, let $\sigma_{0, 2}$ and $\sigma_{2,1}$ be defined as in Remark \ref{rmk:explicit expression}. Then
    \begin{equation} \label{long inequality}
        \sigma_{2,1} \leq \sigma_{0,2}.
    \end{equation}
\end{lemma}
\begin{proof}
The inequality (\ref{long inequality}) is equivalent to prove 
    
   \begin{align} \nonumber
&\left( L^{n+2}(2+Ln)+(n+2L)-\sqrt{\left( L^{n+2}(2+Ln)+(n+2L) \right)^2-8Ln(L^{n+2}-1)^2} \right)(L^{n-2}-1)  \\  \label{ineq: first}
&\leq 2(n-2)(L^{n-1}+1)(L^{n+2}-1).
\end{align}

 Since $(2+Ln)=2+L(n-2) +2L \geq 2+(n-2)+2L=n+2L$. Using this, we have
 \begin{align*}
  \left( L^{n+2}(2+Ln)+(n+2L) \right)^2-8Ln(L^{n+2}-1)^2 & \geq  \left( L^{n+2} + 1 \right)^{2} (n+2L)^{2} - 8 Ln \left( L^{n+2} - 1 \right)^{2} \\
  & \geq \left( L^{n+2} - 1 \right)^{2} \left(  (n+2L)^{2} - 8 Ln \right) \\
  & =  \left( L^{n+2} - 1 \right)^{2}  (2L - n)^{2}.
 \end{align*}
 This gives,
 \begin{align*}
     &\left( L^{n+2}(2+Ln)+(n+2L)-\sqrt{\left( L^{n+2}(2+Ln)+(n+2L) \right)^2-8Ln(L^{n+2}-1)^2} \right)\\
     & \leq (n+2)L^{n+2}+(n-2)L^{n+3}+4L.
 \end{align*}
 Now the inequality \eqref{ineq: first} is true if 
 \begin{align*}
     \left( (n+2)L^{n+2}+(n-2)L^{n+3}+4L\right)(L^{n-2}-1) \leq 2(n-2)(L^{n-1}+1)(L^{n+2}-1)
 \end{align*}
 or equivalently, if
 \begin{align*}
     (n-2)L^{2n+1}-(n+2)L^{2n}+(n-2)L^{n+3}+(3n-2)L^{n+2}-2nL^{n-1}+4L-2(n-2) \geq 0
 \end{align*}
is true. Now the conclusion follows from Lemma (\ref{positive lemma}).
\end{proof} 
       \begin{lemma} \label{lem: inequality2for higher order}
      For $n \geq 3$, $L>1$, let $\sigma_{1, 2}$ and $\sigma_{2,1}$ be defined as in Remark \ref{rmk:explicit expression}. Then 
      \begin{align} \label{2nd inequality}
          \sigma_{2,1} \leq \sigma_{1,2}.
      \end{align}
  \end{lemma} 
  \begin{proof}
        Since $\sqrt{\left((L+n-1)+L^n(1+L(n-1))\right)^2-4L(L^n-1)^2(n-1)} \geq 0$, so the inequality (\ref{2nd inequality}) is true if 
         \begin{align*}
         & \left(L^{n+2}(2+Ln)+(n+2L)-\sqrt{\left( L^{n+2}(2+Ln)+(n+2L) \right)^2-8Ln(L^{n+2}-1)^2}\right)(L^n-1)  \\
         & \leq \left((L+n-1)+L^n(1+L(n-1)\right)(L^{n+2}-1).
         \end{align*}
         or equivalently, if 
         \begin{align*}
&\left( n^2L^{2n+6)}-4nL^{2n+5}+4L^{2n+4}+(2n^2+16n+8)L^{n+3}+4nL^{n+2}+4L^2-4Ln+n^2\right) \\ 
& \times (L^n-1)^2 -\left(L^{2n+2}-(n+1)L^{n+2}+(n+1)L^n-1\right)^2(1+L)^2 \geq 0.
         \end{align*} 
Now consider
\begin{align*}
    \tilde{h}(L)=&\left( n^2L^{2n+6)}-4nL^{2n+5}+4L^{2n+4}+(2n^2+16n+8)L^{n+3}+4nL^{n+2}+4L^2-4Ln+n^2\right) \\
& \times (L^n-1)^2 - \left(L^{2n+2}-(n+1)L^{n+2}+(n+1)L^n-1\right)^2(1+L)^2.
\end{align*}
 Applying the similar technique as in the proof of Lemma \ref{positive lemma}, it can be proved that $\tilde{h}(L)\geq 0$. Hence the conclusion follows.
  \end{proof}
  
  \begin{lemma} \label{lem:Illias}
      For $l\geq 0, i\in\{1, 2\}$, if $\sigma_{l,i}$ are Steklov eigenvalues on annular domain $\Omega_0$. Then
      \begin{enumerate}
          \item the sequence $\{\sigma_{l, 1}\}_{l\geq 0}$ is strictly increasing.
          \item $\sigma_{1, 1} \leq \sigma_{0, 2}$. In particular, the first nonzero Steklov eigenvalue of $\Omega_{0}$ is $\sigma_{1}(\Omega_{0}) = \sigma_{1, 1}$ with multiplicity $n$.
       \end{enumerate}
  \end{lemma}
  For a proof of this Lemma, see [\cite{ftouhi2022place}, Lemma 4.3].

\noindent \textbf{Proof of Theorem \ref{thm:increasing}:} For $n=2$, it follows from Lemma \ref{lem: eigenvalue ineq. for planar} and \ref{lem:Illias} that
\begin{align*}
\sigma_{1, 1} \leq \sigma_{2,1} \leq \sigma_{0, 2} \leq \sigma_{1,2} \quad \text{and} \quad 0 = \sigma_{0, 1} \leq  \sigma_{1, 1} \leq  \sigma_{2, 1} \leq  \sigma_{3, 1} \leq \cdots.
\end{align*}
Combining this with the relation $\sigma_{l,1} < \sigma_{l,2}$ for all $l \geq 0$, we conclude that $\sigma_{2,1}$ is the next smallest eigenvalue after $\sigma_{1,1}$ i.e., $\sigma_{n+1}(\Omega_{0}) = \sigma_{2,1}$. Similarly for $n \geq 3$, $\sigma_{n+1}(\Omega_{0}) = \sigma_{2,1}$ follows from Lemma \ref{lem: inequality1for higher order}, Lemma \ref{lem: inequality2for higher order} and Lemma \ref{lem:Illias}. 
Note that eigenfunctions of \eqref{stk on ball-ball} are of the form $f_{2,1}(r) g_{2}(\omega)$ corresponding to eigenvalue $\sigma_{2,1}$, where $f_{2,1}(r)$ is defined in \eqref{eqn: eigenfunction ODE} and $g_{2}(\omega)$ is an eigenfunction of $\Delta_{\mathbb{S}^{n-1}}$ corresponding to eigenvalue $2n$ i.e., it is an element of the set of all spherical harmonics of degree two. As we already mentioned in the last subsection that multiplicity of $ \sigma_2(\Omega_{0}) = \sigma_{2,1}$ = dim $H_{2} = \frac{(n+2)(n-1)}{2}.$

\section{Some integral inequalities} \label{sec:int. ineq.}

In this section, we provide some results related to function $f_{1,1}(r)$, which will be used in the next section to prove our second main result about higher Steklov eigenvalues, Theorem \ref{thm:isoperimetric}.

\begin{definition}
    A domain  $W\subset \mathbb{R}^n$ is said to be symmetric of order $s$ with respect to the origin, if $R_{i,j}^{\frac{2\pi}{s}}(W)=W$ for all $1\leq i <j \leq n,$ where $R_{i,j}^{\frac{2\pi}{s}}$ denotes anticlockwise rotation with respect to the origin by angle $\frac{2\pi}{s}$ in the coordinate plane $(x_i, x_j).$
\end{definition}

Hereafter, we let $\Omega_{out} \subset \mathbb{R}^n$ be a connected open bounded smooth domain with symmetry of order 4 with respect to the origin and $B_1$ be a unit ball centered at the origin such that $ \Bar{ B_1 }\subset \Omega_{out}$. Denote $\Omega = \Omega_{out} \backslash \Bar{B_1}$. Consider $\Omega_0 = B_L \backslash \Bar{ B_1}$ where $B_L$ is a ball of radius $L$ centered at the origin such that $Vol(B_L)= Vol(\Omega_{out})$ i.e., $Vol(\Omega)= Vol(\Omega_0)$. Let $f_{1,1}:[1,L] \rightarrow \mathbb{R} $ be the function defined as in \eqref{eqn: eigenfunction ODE}. From now onwards, we consider $f_{1,1}:[1,\infty) \rightarrow \mathbb{R} $ by extending function $f_{1,1} (r)$ radially to $[1, \infty)$.

\begin{lemma} \label{decreasing lemma}
Define $F,G : [1,\infty) \rightarrow \mathbb{R}$ as
\begin{align*}
 F(r) = \left((f'_{1,1}(r))^2 + \frac{(n-1)}{r^2} f_{1,1}^2(r)\right)    
\end{align*}
and
\begin{align*}
G(r) = \left( 2f_{1,1}(r)f_{1,1}'(r) + \frac{n-1}{r}f_{1,1}^2(r) \right).   
\end{align*}
Then $F$ is a decreasing function of $r$, and $G$ is an increasing function of $r$.
\end{lemma}
\begin{proof}
    Recall that $f_{1,1}(r)=r + \frac{1+\sigma_{1,1}}{n-1-\sigma_{1,1}}\frac{1}{r^{n-1}}$, then $f_{1,1}'(r)=1-\frac{1+\sigma_{1,1}}{n-1-\sigma_{1,1}}\frac{n-1}{r^n}$. Substituting these values, we get 
    \begin{align*}
        F(r)=& n+\frac{(1+\sigma_{1,1})^2(n-1)n}{(n-1-\sigma_{1,1})^2}\frac{1}{r^{2n}}, \qquad \text{ and } \\
        G(r)=& (n+1)r+\frac{(1+\sigma_{1,1})}{(n-1-\sigma_{1,1})}\frac{2}{r^{(n-1)}}-\frac{(1+\sigma_{1,1})^2(n-1)}{(n-1-\sigma_{1,1})^2}\frac{1}{r^{(2n-1)}}.
    \end{align*}
    By differentiating these functions, we have for all $r \in [1, \infty)$,
    \begin{align*}
        F'(r)=&-\frac{2n^2(n-1)(1+\sigma_{1,1})^2}{(n-1-\sigma_{1,1})^2}\frac{1}{r^{(2n+1)}} \leq 0,  \qquad  \text{ and } \\
     G'(r)=&2+(n-1)\Big (\frac{(1+\sigma_{1,1})}{(n-1-\sigma_{1,1})}\frac{1}{r^n}-1\Big )^2+\frac{2(1+\sigma_{1,1})^2(n-1)^2}{(n-1-\sigma_{1,1})^2}\frac{1}{r^n} \geq 0.
    \end{align*}
    This proves our claim.
\end{proof}

\begin{lemma} \label{lem:integral2}
    Let $F$ be a decreasing radial function defined as in Lemma \ref{decreasing lemma} and $\Omega, \Omega_0$ be defined as above. Then the following inequality holds.
    \begin{equation}
      \int_\Omega F(r) dV \leq \int_{\Omega_0} F(r) dV.
    \end{equation}
\end{lemma}
For a proof, see \cite{basak2023sharp}.

\begin{lemma}\label{integral6}
   Let $f_{1,1}(r), \Omega_{out}$ and $B_L$ be defined as above, and $\partial \Omega_{out}, \partial B_L$ are boundaries of $\Omega_{out}, B_L,$ respectively. Then the following inequality holds.
   \begin{equation} \label{inequality 5}
      \int_{\partial {\Omega_{out}}} f_{1,1}^2(r) dS \geq  \int_{\partial {B_L}} f_{1,1}^2(r) dS.
   \end{equation}
\end{lemma} 
\begin{proof}
     Let $B= \{ R_u u \,| \,\, u\in
        S^{n-1}\}$, where $S^{n-1}$ is $(n-1)$ dimensional unit sphere and $R_u= sup\{r \ | \ r u\in \partial \Omega \}$. Then 
        \begin{align*}
             \int_{\partial {\Omega_{out}}} f_{1,1}^2(r) dS & \geq \int_{B} f_{1,1}^2(r) dS \\
             &= \int_{u \in {S} ^{n-1}} f_{1,1}^2(R_u) sec(\theta) R_u^{n-1} du  \\
             & \geq \int_{ u\in {S} ^{n-1}} f_{1,1}^2(R_u) R_u^{n-1} du \\
             & = \int_{u \in S^{n-1}}\int_{1}^{R_u} \left ( 2 f_{1,1}(r) f'_{1,1}(r) r^{n-1} + f_{1,1}^2(r) (n-1) r^{n-2} \right) dr du +f_{1,1}^2(1)|S^{n-1}|\\
             & = \int_{u\in {S}^{n-1}}\int_{1}^{R_u} \left( 2 f_{1,1}(r) f'_{1,1}(r)+ f_{1,1}^2(r) \frac{(n-1)}{r}  \right)r^{n-1} dr du + f_{1,1}^2(1)|S^{n-1}|\\
             & \geq \int_{\Omega} \left( 2 f_{1,1}(r) f'_{1,1}(r) + f_{1,1}^2(r) \frac{(n-1)}{r}\right) dV +f_{1,1}^2(1)|S^{n-1}| \\
             &=\int_{\Omega}G(r)dV+f_{1,1}^2(1)|S^{n-1}|.
        \end{align*}
        
        where $G(r) = \left( 2 f_{1,1}(r) f'_{1,1}(r)  + f_{1,1}^2(r) \frac{(n-1)}{r} \right)$ is defined as in lemma (\ref{decreasing lemma}). Since $G$ is an increasing function of $r$, we have
        \begin{equation} \label{inequality 3}
            G(r) < G(L) \,\,\, \text{in}\,\,  \Omega_0 \backslash (\Omega \cap \Omega_0), 
           ~\mbox{ and } ~~ G(r) > G(L)\,\, \text{in} \,\,\,  \Omega \backslash (\Omega \cap \Omega_0).
       \end{equation}
        Now 
         \begin{align*}
            \int_ { \Omega} G(r) dV +f_{1,1}^2(1)|S^{n-1}|& = \int_{\Omega \cap \Omega_0} G(r) dV+ \int_{\Omega \backslash (\Omega \cap \Omega_0)} G(r) dV +f_{1,1}^2(1)|S^{n-1}|\\
            & =  \int_{ \Omega_0} G(r) dV  - \int_{ \Omega_0 \backslash (\Omega \cap \Omega_0)} G(r) dV + \int_{ \Omega \backslash (\Omega \cap \Omega_0)} G(r) dV +f_{1,1}^2(1)|S^{n-1}|.
        \end{align*} 
        Thus, from inequality (\ref{inequality 3}), we get 
        \begin{align*}
            \int_{ \Omega} G(r) dV +f_{1,1}^2(r)|S^{n-1}| \geq \int_{\Omega_0} G(r) dV  - \int_{\Omega_0 \backslash (\Omega \cap \Omega_0)} G(L) dV + \int_{\Omega \backslash (\Omega \cap \Omega_0)} G(L) dV +f_{1,1}^2(1)|S^{n-1}|.
        \end{align*}
         Since Vol$(\Omega_0 \backslash (\Omega \cap \Omega_0))$ = Vol$(\Omega \backslash (\Omega \cap \Omega_0))$, we get
         \begin{align*}
             \int_{\Omega} G(r) dV +f_{1,1}^2(r)|S^{n-1}| & \geq  \int_{\Omega_0} G(r) dV +f_{1,1}^2(1)|S^{n-1}| \\
              &= \int_{\Omega_0}  \left( 2 f_{1,1}(r) f'_{1,1}(r)  + f_{1,1}^2(r) \frac{(n-1)}{r}\right) dV  +f_{1,1}^2(1)|S^{n-1}|\\
             & = \int_{u \in S^{n-1}} \int_{r=1}^{L} \left( 2 f_{1,1}(r) f'_{1,1}(r)  + f_{1,1}^2(r) \frac{(n-1)}{r}  \right) r^{n-1} drdu+f_{1,1}^2(1)|S^{n-1}|\\
             &= \int_{u\in S^{n-1}} \Big[f_{1,1}^2(r)r^{n-1} \Big]_{1}^L du + f_{1,1}^2(1)|S^{n-1}| \\
             & = \int_{u\in S^{n-1}} \left(f_{1,1}^2(L) L^{n-1} - f_{1,1}^2(1)\right) du + f_{1,1}^2(1)|S^{n-1}| \\
             & = \int_{\partial B_{L}} f_{1,1}^2(r) dS.
         \end{align*}
         Thus,
             $\displaystyle \int_{\partial {\Omega_{out}} }f_{1,1}^2(r) dS \geq \int_{\partial B_{L}} f_{1,1}^2(r) dS$.
\end{proof}
\begin{cor} \label{cor: integral4}
  The following inequality holds for $ \Omega, \Omega_0 $ and $f_{1,1}$:
    \begin{equation}
        \displaystyle \int_{\partial {\Omega }}f_{1,1}^2(r) dS \geq \int_{\partial \Omega_0} f_{1,1}^2(r) dS
    \end{equation}
\end{cor}
We will use the following proposition to conclude some integral identities related to function $f_{1,1}(r)$, see Corollary \ref{cor: f_{1,1}}.
\begin{proposition} \label{prop:integral expression}
 Let $g : (0, \infty) \rightarrow \mathbb{R}$ be a smooth positive radial function of $r$. Let $W$  be a bounded smooth domain in $\mathbb{R}^n$ with smooth boundary $\partial W$.
 \begin{enumerate}
     \item If $W$ is symmetric of order 2 with respect to the origin, then for each $i=1,2,\ldots n$, we have
     \begin{enumerate}
         \item $\displaystyle \int_{x \in W}   g(r) x_i\, dV = 0$,  \label{eqn:integral 1}~~~
       \item $\displaystyle \int_{x \in \partial W}   g(r) x_i\, dS = 0$.  \label{eqn: integral 3}.
        \end{enumerate}
\item If $W$ is  symmetric of order 4 with respect to the origin, then for each $i, j=1,2,\dots n$, $i\neq j$,  we have
\begin{enumerate}
     \item $\displaystyle \int_{ x \in W} g(r) x_i x_j \, dV = 0$, \label{eqn:integral 2}
      \item $\displaystyle \int_{x \in \partial W}  g(r) x_i x_j \,dS = 0$.  \label{eqn: integral 4}
        \end{enumerate}
 \end{enumerate}
 Here $(x_1, x_2, \ldots, x_n)$ represents Cartesian coordinates of a point $x \in \mathbb{R}^{n}$ with respect to the origin.
\end{proposition}

\begin{cor} \label{cor: f_{1,1}}
    Let $\Omega$ and $f_{1,1}$ be defined as in the beginning of this section. Then for each $i,j=1,2,\dots n$, $i\neq j$ we have 
    \begin{enumerate}
        \item $\displaystyle \int_{x\in \partial \Omega}f_{1,1}(r) \frac{f_{1,1}(r)}{r}x_i\, dS = 0,$ \label{proof (i)}\\
         \item $\displaystyle \int_{x \in \partial \Omega}  \frac{f_{1,1}(r)}{r}x_i. \frac{f_{1,1}(r)}{r}x_j \,dS=0,$
        \item $\displaystyle \int_{ x \in \Omega} \left< \nabla f_{1,1}(r), \nabla \left(\frac{f_{1,1}(r)}{r}x_i \right)\right> dV=0,$ 
        \item $\displaystyle \int_{x \in \Omega} \left< \nabla \left(\frac{f_{1,1}(r)}{r} x_i \right) , \nabla \left(\frac{f_{1,1}(r)}{r}x_j \right)\right> dV=0$. 
    \end{enumerate}
\end{cor}

\begin{lemma} \label{lem:integral1}
    Let $W \subset \mathbb{R}^n$ be an open bounded smooth domain having symmetry of order $4$ with respect to the origin. Let $\Phi(r)$ be a positive radial function on $\mathbb{R}^n$. Then, there exists a constant $A>0$ such that
    \begin{equation}
         \int_{W} \Phi(r) x_i^2\,\, dV= A\,\, \text{for all}\, \,i\in \{1,2 \dots n\}.
    \end{equation}
\end{lemma}
  
 For a detailed proof of Proposition \ref{prop:integral expression}, Corollary \ref{cor: f_{1,1}} and Lemma \ref{lem:integral1}, see (\cite{basak2023sharp}).

\section{Steklov eigenvalue problem on symmetric doubly connected domain} \label{sec: iso. bound}
Let $\Omega_{out} \subset \mathbb{R}^n$ be a connected open bounded smooth domain with symmetry of order 4 centered at the origin. Let $B_1$ be a unit ball centered at the origin such that $ \Bar{ B_1 }\subset \Omega_{out}$. 
Consider the following Steklov eigenvalue problem on $\Omega = \Omega_{out} \backslash \Bar{B_1}$.
\begin{align} 
    \begin{cases} 
        \Delta u =0 \quad in \quad \Omega, \\
        \frac{\partial u}{\partial \nu}= \sigma u \quad on \quad \partial \Omega.
    \end{cases} \label{problem on annular domain }
\end{align} 
For $k \in \mathbb{N}, \sigma_{k}$ admits the following variational characterization
\begin{align}  \label{variational characterization}
\sigma_{k}(\Omega)=\min_{E\in \mathcal{H}_{k+1}(\Omega)} \max_{u(\neq 0) \in E} R(u),
\end{align} 

where $\mathcal{H}_{k+1}(\Omega)$ is the collection of all $(k+1)$ dimensional subspace of the sobolev space $H^1(\Omega)$ and $R(u):= \frac{\displaystyle \int_{\Omega} \| \nabla u \|^2 dV}{\displaystyle \int_{\partial{\Omega}} \|u\|^2 dS}$.\\

Now we prove the following theorem for the first $n$ nonzero Steklov eigenvalues.
\begin{theorem} \label{thm:isoperimetric}
    Let $\Omega$ be a connected open bounded smooth domain defined as above, and $\sigma_{k}$ is the $k$th eigenvalue of (\ref{problem on annular domain }) on $\Omega$. Then for $1 \leq k \leq n$, 
    \begin{equation}
        \sigma_{k}(\Omega) \leq \sigma_{k}(\Omega_0) = \sigma_1 (\Omega_0), 
    \end{equation} 
    where $\Omega_0 = B_L \backslash \Bar{ B_1}$ and $B_L$ is a ball of radius $L$ centered at the origin such that $Vol(B_L)= Vol(\Omega_{out})$ i.e, $Vol(\Omega)= Vol(\Omega_0)$.
    \end{theorem}
    \begin{proof}
        Since $\sigma_k(\Omega_0)=\sigma_1(\Omega_0), 1\leq k \leq n $, so it is enough to prove that $\sigma_n(\Omega) \leq \sigma_n(\Omega_0)=\sigma_1(\Omega_0)$. Consider the following $(n+1)$ dimensional subspace of $H^1(\Omega)$,
        \begin{align*}
            E=span\{ f_{1,1}(r), \frac{f_{1,1}(r)}{r}x_1, \dots \frac{f_{1,1}(r)}{r}x_n \}, 
        \end{align*}
        where $f_{1,1}(r)$ is define in (\ref{eqn: eigenfunction ODE}) with $l=1$ and $i=1$. For any $u \in E \backslash \{0\}$, there exist $c_0, c_1, \dots c_n \in \mathbb{R}$ not simultaneously equal to zero, such that 
        \begin{align*}
            u= c_0 f_{1,1}(r) + c_1 \frac{f_{1,1}(r)}{r}x_1 + \dots + c_n\frac{f_{1,1}(r)}{r}x_n.
        \end{align*}
        Then by using Corollary \ref{cor: f_{1,1}}, we get 
        \begin{align}
            \frac{\displaystyle \int_\Omega \| \nabla u \|^2 dV}{\displaystyle \int_{\partial {\Omega}} u^2 dS}= \frac{c_0^2 \displaystyle \int_\Omega \|\nabla f_{1,1}(r)\|^2 \, dV+ \displaystyle \sum_{i=1}^nc_i^2 \displaystyle \int_\Omega \bigg \| \nabla \left( \frac{f_{1,1}(r)}{r}  x_i \right) \bigg \|^2 \, dV}{c_0^2 \displaystyle \int_{\partial {\Omega}}f_{1,1}^2(r) \, dS + \displaystyle \sum_{i=1}^n c_i^2 \displaystyle \int_{\partial {\Omega}} \frac{f_{1,1}^2(r)}{r^2}x_i^2 \, dS}. \label{equality}
        \end{align}
        According to Lemma \ref{lem:integral1} there are constants $A_1, A_2 > 0 $ such that for all natural numbers $ 1\leq i \leq n$,
         \begin{align*}
       \int_{\partial {\Omega}}\left( \frac{f_{1,1}(r)}{r}x_i \right)^2 dS & =\int_{\partial {\Omega}} \frac{f_{1,1}^2(r)}{r^2}x_i^2 dS = A_1,\\
       \int_\Omega \bigg \| \nabla \left( \frac{f_{1,1}(r)}{r} x_i \right) \bigg \|^2 dV & = \int_\Omega \left( \frac{(f'_{1,1}(r))^2}{r^2} x_i^2 - \frac{f_{1,1}^2(r)}{r^4}x_i^2 + \frac{f_{1,1}^2(r)}{r^2} \right) dV = A_2.
   \end{align*} 
   Therefore
   $$ n\,A_1 = \sum_{i=1}^n \int_{\partial {\Omega}} \left( \frac{f_{1,1}(r)}{r}x_i \right)^2 \, dS = \int_{\partial {\Omega}} f_{1,1}^2(r) \, dS, $$ and
 $$n\,A_2 = \sum_{i=1}^n  \int_\Omega \left(\frac{(f'_{1,1}(r))^2}{r^2} x_i^2-\frac{f_{1,1}^2(r)}{r^4}x_i^2 +\frac{f_{1,1}^2(r)}{r^2} \right) dV =\int_\Omega \left((f'_{1,1}(r))^2 + \frac{(n-1)}{r^2}f_{1,1}^2(r)\right)\, dV.$$
 Thus for all natural numbers $ 1\leq i\leq n,$ we have 
  \begin{equation} \label{sum 1} 
       \displaystyle \int_{\partial {\Omega}} \left( \frac{f_{1,1}(r)}{r}x_i \right)^2 \, dS  = A_1 = \frac{1}{n} \displaystyle \int_{\partial {\Omega}} f_{1,1}^2(r) \, dS.
        \end{equation}
         \begin{equation} \label{sum 1.5} 
        \displaystyle \int_\Omega \bigg \| \nabla \left( \frac{f_{1,1}(r) x_i}{r}\right) \bigg \|^2 \, dV  = A_2 = \frac{1}{n} \displaystyle \int_\Omega \left( (f'_{1,1}(r))^2 + \frac{(n-1)}{r^2} f_{1,1}^2(r)\right) \, dV.
         \end{equation}
         Now, from (\ref{equality}), \eqref{sum 1} and  \eqref{sum 1.5}, we get 
   \begin{align}
       \frac{\displaystyle \int_\Omega \| \nabla u \|^2 dV}{\displaystyle \int_{\partial {\Omega}} u^2 \, dS} = \frac{c_0^2 \displaystyle \int_\Omega \|\nabla f_{1,1}(r)\|^2 \, dV + A_2 \displaystyle \sum_{i=1}^n c_i^2}{c_0^2 \displaystyle \int_{\partial {\Omega}}f_{1,1}^2(r) \, dS + A_1 \displaystyle \sum_{i=1}^n c_i^2 }
        \leq \text{max} \Bigg \{\frac{\displaystyle \int_\Omega \|\nabla f_{1,1}(r)\|^2 \, dV}{\displaystyle \int_{\partial {\Omega}}f_{1,1}^2(r) \, dS}, \frac{A_2}{A_1} \Bigg \}. \label{gradient inequality 1}
   \end{align}
   Now 
\begin{align*}
    \frac{A_2}{A_1} & = \frac{\displaystyle \int_\Omega \left( (f'_{1,1}(r))^2 + \frac{(n-1)}{r^2} f_{1,1}^2(r)\right) dV}{\displaystyle \int_{\partial {\Omega}} f_{1,1}^2(r) \, dS} 
                     \geq \frac{\displaystyle \int_\Omega (f'_{1,1}(r))^2 \, dV}{\displaystyle \int_{\partial {\Omega}} f_{1,1}^2(r) \, dS} 
                     = \frac{\displaystyle \int_\Omega \|\nabla f_{1,1}(r)\|^2 \, dV}{\displaystyle \int_{\partial {\Omega}}f_{1,1}^2(r) \, dS}.
\end{align*}
   Then from the inequality (\ref{gradient inequality 1}) we get
   \begin{align}
       \frac{\displaystyle \int_\Omega \| \nabla u \|^2 dV}{\displaystyle \int_{\partial {\Omega}} u^2 \, dS} \leq  \frac{A_2}{A_1} =\frac{\displaystyle \int_\Omega \left( (f'_{1,1}(r))^2 + \frac{(n-1)}{r^2} f_{1,1}^2(r)\right) \, dV }{\displaystyle \int_{\partial {\Omega}} f_{1,1}^2(r) \, dS}. \label{gradient inequality}
    \end{align}
    Next, using the Lemma \ref{lem:integral2} and Corollary \ref{cor: integral4}, we get  
   \begin{align*}
       \frac{ \displaystyle \int_\Omega \left( (f_{1,1}'(r))^2 + \frac{(n-1)}{r^2} f_{1,1}^2(r)\right) \, dV }{\displaystyle \int_{\partial {\Omega}} f_{1,1}^2(r) \, dS} \leq \frac{\displaystyle \int_{\Omega_{0}} \left( (f_{1,1}'(r))^2 + \frac{(n-1)}{r^2} f_{1,1}^2(r)\right ) \, dV }{\displaystyle \int_{\partial \Omega_0}f_{1,1}^2(r) \, dS} = \sigma_1(\Omega_0).
\end{align*}
Therefore, from the variational characterization (\ref{variational characterization}) and inequality (\ref{gradient inequality}), we conclude,
   \begin{align*}
      \sigma_{n}(\Omega) \leq \max_{u(\neq 0) \in E} \frac{\displaystyle \int_\Omega \| \nabla u \|^2 \, dV}{\displaystyle \int_{\partial {\Omega}} u^2 \, dS} \leq \sigma_1(\Omega_0).
   \end{align*}
    \end{proof}

    \section{Mixed Steklov Neumann eigenvalue problem } \label{sec:SN problem}
  In this section, we study the mixed Steklov Neumann problem on doubly connected domains. Throughout this section, we consider the following notations: Let $\Omega_{out}$ denote a connected smooth bounded domain in $\mathbb{R}^n$ centered at the origin having symmetric of order $4$ and $B_{R_1}$ is a ball of radius $R_1$ in $\mathbb{R}^n$ centered at the origin such that $\bar{B}_{R_1} \subset \Omega_{out}$. Take $\tilde{\Omega}_0$ as a concentric annular domain in $\mathbb{R}^{n}$. We consider the following Steklov Neumann eigenvalue problem

  \begin{align} \label{SNproblem1}
    \begin{cases}
        \Delta u = 0 \,\, \text{in}\,\,\, \tilde{\Omega}=\Omega_{out}\backslash \bar{B}_{R_1}, \\
     \frac{\partial u}{\partial \nu} = 0 \,\, \text{on}\,\,\, \partial B_{R_1}, \\
     \frac{\partial u}{\partial \nu} = \mu u \,\, \text{on}\,\,\, \partial \Omega_{out}. 
    \end{cases}
\end{align}
Eigenvalues of problem (\ref{SNproblem1}) are positive and form an increasing sequence 
\begin{align*}
    0=\mu_0(\tilde{\Omega})<\mu_1(\tilde{\Omega}) \leq \mu_2(\tilde{\Omega}) \leq \cdots  \nearrow \infty
\end{align*}
Here eigenvalues are counted with multiplicity.

\subsection{Steklov Neumann Eigenvalue problem on concentric annular domain}

For $R_2 >R_1,$ let $B_{R_1}$ and $B_{R_2}$ be balls in $\mathbb{R}^n$ centered at the origin of radius $R_1$ and $R_2$, respectively. Consider the Steklov Neumann eigenvalue problem on an annular domain $\tilde{\Omega}_0 = B_{R_2} \backslash B_{R_1},$

\begin{align} \label{SN problem}
    \begin{cases}
        \Delta u = 0 \,\, \text{in}\,\,\, \tilde{\Omega}_0, \\
     \frac{\partial u}{\partial \nu} = 0 \,\, \text{on}\,\,\, \partial B_{R_1}, \\
     \frac{\partial u}{\partial \nu} = \mu u \,\, \text{on}\,\,\, \partial B_{R_2}. 
    \end{cases}
\end{align}
In a similar manner as described in Section \ref{sec:annular domain}, using separation of variable technique, any solution of (\ref{SN problem}) is of the form $u=f(r)g(w)$. Here $g(w)$ is an eigenfunction of $\Delta_{\mathbb{S}^{n-1}}$ and $f(r)$ satisfies 
\begin{eqnarray} 
    \begin{array}{ll} \label{SNode}
        -f''(r)-\frac{n-1}{r}f'(r)+ \frac{l(l+n-2)}{r^2}f(r)=0 ~\mbox{ for } ~ r \in (R_1,R_2), l\geq 0, \\
    f'(R_1)=0, \,\, f'(R_2)=\mu f(R_2).
    \end{array}
\end{eqnarray}
By solving this ODE, we find that corresponding to each $l \geq 0$, it has one eigenvalue $\mu_l $ which is given by
\begin{align*}
    \mu_l = \frac{l(l+n-2) \Big((\frac{R_2}{R_1})^{2l+n-2}-1\Big)}{R_2 \Big( (l+n-2)(\frac{R_2}{R_1})^{2l+n-2}+l\Big)}, \quad l \geq 0,
\end{align*}
corresponding to eigenfunction
\begin{align}\label{Neuman eigenfunction}
    \tilde{f}_l(r)=r^l+\frac{lR_1^{2l+n-2}}{(l+n-2)r^{l+n-2}}, \quad l \geq 0.
\end{align}

Hence for each $l\geq 0$, $\mu_l (\tilde{\Omega}_0) = \mu_l$ is an eigenvalue of \eqref{SN problem} with multiplicity equal to the dimension of $H_l$. Eigenfunctions corresponding to eigenvalue $\mu_l (\tilde{\Omega}_0)$ are $\tilde{f}_{l}(r)g_{l}(w)$, where $g_{l}(w)$ is an eigenfunction of $\Delta_{\mathbb{S}^{n-1}}$ corresponding to eigenvalue $l(l+n-2)$ and $\tilde{f}_{l}(r)$ is defined as in \eqref{Neuman eigenfunction}. Also
\begin{align}
  \mu_1 (\tilde{\Omega}_0) = \mu_2 (\tilde{\Omega}_0) = \cdots = \mu_n (\tilde{\Omega}_0) = \frac{\displaystyle \int_{\tilde{\Omega}_{0}} \left( (\tilde{f}_{1}'(r))^2 + \frac{(n-1)}{r^2} \tilde{f}_{1}^2(r)\right ) \, dV }{\displaystyle \int_{\partial B_{R_2}}\tilde{f}_{1}^2(r) \, dS}.
\end{align}

\begin{lemma} \label{SN de(increasing)lemma}
  Let $\tilde{f}_1:[R_1, \infty) \to \mathbb{R}$ be function define in \eqref{Neuman eigenfunction} for $l=1$. Define $\tilde{F}, \tilde{G}:[R_1, \infty) \to \mathbb{R}$ as 
    \begin{align*}
        \tilde{F}(r) &=\left((\tilde{f}'_{1}(r))^2 + \frac{(n-1)}{r^2} \tilde{f}_{1}^2(r)\right) \qquad \text{ and } \\
        \tilde{G}(r) &= \left( 2\tilde{f}_{1}(r)\tilde{f}_{1}'(r) + \frac{n-1}{r}\tilde{f}_{1}^2(r) \right).
    \end{align*}

Then $\tilde{F}$ is a decreasing function of $r$ and $\tilde{G}$ is an increasing function of $r$. 
\end{lemma}

\begin{proof}
    We have $\tilde{f}_1(r)=r+\frac{R_1^n}{(n-1)r^{n-1}}$ and $\tilde{f}_1'(r)=1-\frac{R_1^n}{r^n}.$ Substituting these values, we get 
\begin{align*}
   \tilde{F}(r)&=n+\frac{n}{n-1}\frac{R_1^{2n}}{r^{2n}} \,\, \,\quad \text{and}\\
   \tilde{G}(r)&=(n+1)r+\frac{2R_1^n}{(n-1)r^{n-1}}-\frac{R_1^{2n}}{(n-1)r^{2n-1}}.
\end{align*} 
Then by differentiating these functions, we have for all $r\in(R_1, \infty)$
\begin{align*}
    \tilde{F}'(r)&=-\frac{2n^2}{n-1}\frac{R_1^{2n}}{r^{2n+1}}\leq 0, \quad \text{and}\\
    \tilde{G}'(r)&=(n+1)- \frac{2R_1^n}{r^n}+\frac{2n-1}{n-1}\frac{R_1^{2n}}{r^{2n}}
   \geq \frac{(r^n-R_1^n)^2}{r^{2n}}\geq 0.
\end{align*}
Hence the conclusion follows.
\end{proof}

\subsection{Mixed Steklov Neumann problem on symmetric doubly connected domain}
In this subsection, we prove a sharp upper bound of the first $n$ nonzero mixed Steklov Neumann eigenvalues on $\tilde{\Omega}=\Omega_{out}\backslash \bar{B}_{R_1}.$ Let $\tilde{\Omega}_0 = B_{R_2} \backslash \Bar{B}_{R_1}$, where $B_{R_2}$ is a ball in $\mathbb{R}^{n}$ of radius $R_2$ centered at the origin. We choose $R_2 > R_1$ such that $Vol(B_{R_2})= Vol(\Omega_{out})$ i.e., $Vol(\Omega)= Vol(\tilde{\Omega}_0)$.

It is well known that for each $k\in \mathbb{N}, \mu_k(\tilde{\Omega})$ admits the following variational characterization
\begin{align}  \label{variational characterization (SN)}
\mu_{k}(\tilde{\Omega})=\min_{E\in \mathcal{H}_{k+1}(\Omega)} \max_{u(\neq 0) \in E} \tilde{R}(u),
\end{align} 
where $\mathcal{H}_{k+1}(\tilde{\Omega})$ is the collection of all $(k+1)$ dimensional subspace of the sobolev space $H^1(\tilde{\Omega})$ and $\tilde{R}(u)= \frac{\displaystyle \int_{\tilde{\Omega}} \| \nabla u \|^2 dV}{\displaystyle \int_{\partial{\Omega_{out}}} \|u\|^2 dS}$.

\begin{lemma} \label{SN integral 1}
    
          Let $\tilde{F}$ be a decreasing radial function defined as in Lemma \ref{SN de(increasing)lemma} and $\tilde{\Omega}, \tilde{\Omega_0}$ be defined as above. Then the following inequality holds.
    \begin{equation*}
         \int_{\tilde{\Omega}} \tilde{F}(r) dV \leq \int_{\tilde{\Omega}_0} \tilde{F}(r) dV.
    \end{equation*}    
\end{lemma}
 For proof, see \cite{basak2023sharp}.


\begin{lemma} \label{SN integral 2}
     Let $\tilde{f}_{1}(r), {\Omega}_{out} $ and $B_{R_2}$ be defined as above, and $\partial {\Omega}_{out}, \partial B_{R_2}$ are boundaries of $\Omega_{out}, B_{R_2},$ respectively. Then the following inequality holds.
   \begin{equation} \label{inequality 5}
      \int_{\partial {\Omega_{out}}} \tilde{f}_{1}^2(r) dS \geq  \int_{\partial {B_{R_2}}} \tilde{f}_{1}^2(r) dS.
   \end{equation}
\end{lemma}
The proof will follow the same as the proof of Lemma \ref{integral6}.
 
Now we state the main result related to mixed Steklov Neumann eigenvalues.
\begin{theorem} \label{thm:isoperimetricSN}
    Let $\tilde{\Omega}$ and $\tilde{\Omega}_0$ be defined as the above, and $\mu_k$ is the $k$th eigenvalue of (\ref{SNproblem1}) on $\tilde{\Omega}.$ Then for each $1\leq k \leq n,$
    \begin{equation}
        \mu_k(\tilde{\Omega}) \leq \mu_k(\tilde{\Omega}_0)=\mu_1(\tilde{\Omega}_0).
    \end{equation}
\end{theorem}

\begin{proof}
 Since $\mu_k(\tilde{\Omega}_0)=\mu_1(\tilde{\Omega}_0), 1\leq k \leq n $, so it is enough to prove that $\mu_n(\tilde{\Omega}) \leq \mu_n(\tilde{\Omega}_0)=\mu_1(\tilde{\Omega}_0)$. To find bound for $\mu_n(\tilde{\Omega})$, we consider an $(n+1)$ dimensional subspace of $H^1(\tilde{\Omega})$ and use it in the variational characterization \eqref{variational characterization (SN)} of $\mu_n(\tilde{\Omega})$. Take
        \begin{align*}
            E = span \bigg\{\tilde{f}_{1}(r), \frac{\tilde{f}_{1}(r)}{r}x_1, \dots \frac{\tilde{f}_{1}(r)}{r}x_n \bigg\}, 
        \end{align*}
        where $\tilde{f}_{1}(r)$ is define in (\ref{Neuman eigenfunction}) with $l=1$. For any $u \in E \backslash \{0\}$, there exist $d_0, d_1, \dots d_n \in \mathbb{R}$ not simultaneously equal to zero, such that 
        \begin{align*}
            u= d_0 \tilde{f}_{1}(r) + d_1 \frac{\tilde{f}_{1}(r)}{r}x_1 + \dots + d_n\frac{\tilde{f}_{1}(r)}{r}x_n.
        \end{align*}
        Then by using proposition \ref{prop:integral expression}, we get 
        \begin{align}
            \frac{\displaystyle \int_{\tilde{\Omega}} \| \nabla u \|^2 dV}{\displaystyle \int_{\partial {\Omega_{out}}} u^2 dS}= \frac{d_0^2 \displaystyle \int_{\tilde{\Omega}} \|\nabla \tilde{f}_{1}(r)\|^2 \, dV+ \displaystyle \sum_{i=1}^nd_i^2 \displaystyle \int_{\tilde{\Omega}} \bigg \| \nabla \left( \frac{\tilde{f}_{1}(r)}{r}  x_i \right) \bigg \|^2 \, dV}{d_0^2 \displaystyle \int_{\partial {\Omega_{out}}}\tilde{f}_{1}^2(r) \, dS + \displaystyle \sum_{i=1}^n d_i^2 \displaystyle \int_{\partial {\Omega_{out}}} \frac{\tilde{f}_{1}^2(r)}{r^2}x_i^2 \, dS}. \label{equality}
        \end{align}

Then using Lemma \ref{lem:integral1} and the similar argument used in the proof of Theorem \ref{thm:isoperimetric}, we get
\begin{align}
       \frac{\displaystyle \int_{\tilde{\Omega}} \| \nabla u \|^2 dV}{\displaystyle \int_{\partial {\Omega_{out}}} u^2 \, dS} \leq  \frac{\displaystyle \int_{\tilde{\Omega}} \left( (\tilde{f}'_{1}(r))^2 + \frac{(n-1)}{r^2} \tilde{f}_{1}^2(r)\right) \, dV }{\displaystyle \int_{\partial {\Omega_{out}}} \tilde{f}_{1}^2(r) \, dS}. \label{gradient inequality}
    \end{align}
Next, using Lemma \ref{SN integral 1} and Lemma \ref{SN integral 2} we get

 \begin{align*}
       \frac{ \displaystyle \int_{\tilde{\Omega}} \left( (\tilde{f}_{1}'(r))^2 + \frac{(n-1)}{r^2} \tilde{f}_{1}^2(r)\right) \, dV }{\displaystyle \int_{\partial {\Omega_{out}}} \tilde{f}_{1}^2(r) \, dS} \leq \frac{\displaystyle \int_{\tilde{\Omega}_{0}} \left( (\tilde{f}_{1}'(r))^2 + \frac{(n-1)}{r^2} \tilde{f}_{1}^2(r)\right ) \, dV }{\displaystyle \int_{\partial B_{R_2}}\tilde{f}_{1}^2(r) \, dS} = \mu_1(\tilde{\Omega}_0).
\end{align*}
This proves the theorem.
\end{proof}

\begin{figure}
    \centering

    \begin{minipage}{0.4\textwidth}
        \centering
        \includegraphics[width=\textwidth]{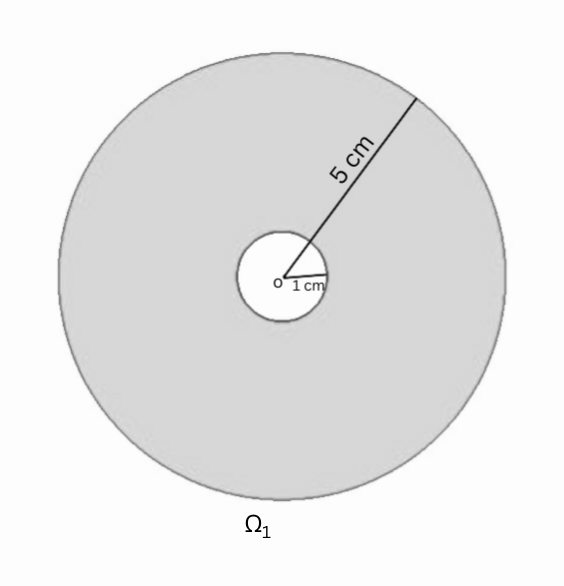} 
        \caption{Concentric Annular Domain}
         \label{fig:figure1}
    \end{minipage}%
   \hspace{0.5 cm}
    \begin{minipage}{0.5\textwidth}
        \centering
        \vspace{1cm} 
        \begin{tabular}{|c|c|c|c|}
            \hline
            Domain $(\Omega)$ &$ \Omega_1$& $\Omega_2$ &$\Omega_3 $\\ \hline
           $ \sigma_2 $   & $ 0.1783 $  &$ 0.2384$   & $ 0.20204$   \\ \hline
            $\mu_2$  & $0.18467 $  & $0.23222$   &$ 0.24484$   \\ \hline
            
        \end{tabular}
        \captionof{table}{Comparison of eigenvalues}
        \label{tab:mytable}
    \end{minipage}

    \begin{minipage}{0.4\textwidth}
        \centering
        \includegraphics[width=\textwidth]{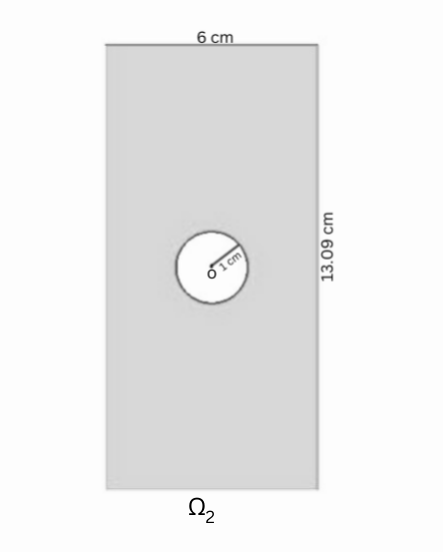}
        \caption{Rectangle with spherical hole} 
        \label{fig:rectangle-ball}
        \end{minipage}
    \hfill
    \begin{minipage}{0.4\textwidth}
        \centering
        \includegraphics[width=\textwidth, height = 8.5 cm ]{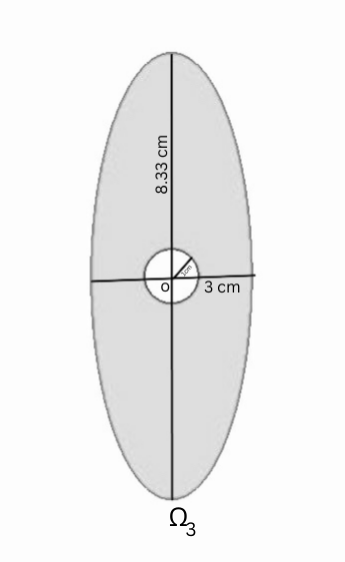}
      \caption{Ellipse with spherical hole}
        \label{fig:ellipse-ball}
        \end{minipage}
    

    
\end{figure}

\section{Counter examples} \label{Counter examples}
In this section, we provide some examples to emphasis the fact that assumption `symmetry of order $4$' on domains is crucial. In particular, we present some domains having symmetry of order $2$ (not having symmetry of order $4$) for which Theorem \ref{thm:isoperimetric} and \ref{thm:isoperimetricSN} fails. Consider concentric annular domain $\Omega_1$ (Figure \ref{fig:figure1}) with outer and inner radius of 5 cm and 1 cm, respectively.
\begin{enumerate}
    \item Take domain $\Omega_2$, rectangle (with sides $13.095$ cm and $6$ cm) having a spherical hole of radius 1, see Figure \ref{fig:rectangle-ball}. Then using FreeFEM++ \cite{hecht2012new}, we find that $\sigma_{2}(\Omega_2) > \sigma_{2}(\Omega_1)$ and $\mu_{2}(\Omega_2) > \mu_{2}(\Omega_1)$ (see Table \ref{tab:mytable}).
    \item Our next example is domain $\Omega_3$, an ellipse (with 8.33 cm major axis and 3 cm minor axis) having a spherical hole of radius 1 (Figure \ref{fig:ellipse-ball}). Then again using FreeFEM++ \cite{hecht2012new}, we obtain that $\sigma_{2}(\Omega_3) > \sigma_{2}(\Omega_1)$ and $\mu_{2}(\Omega_3) > \mu_{2}(\Omega_1)$ (see Table \ref{tab:mytable}).
\end{enumerate}
These examples show that symmetry condition in Theorems \ref{thm:isoperimetric} and \ref{thm:isoperimetricSN} can not be dropped.

\section{Numerical Observations} \label{sec:numerical observation}
Let $\Omega_{(t_1, t_2)} = \Omega_{out} \setminus B_1(t_1, t_2)$ be a doubly connected domain, where $B_1(t_1, t_2)$ is the ball of radius $1$ centered at $(t_1, t_2)$ and $B_1(t_1, t_2)$ is compactly contained in $\Omega_{out}$. In this section, we present numerical values of first and second nonzero eigenvalues of Steklov problem and mixed Steklov Neumann problem on following two different types of doubly connected domains:
\begin{enumerate}
    \item When outer domain is a ball and $B_1(t_1, t_2)$ varies inside the ball.
    \item When outer domain is an ellipse and $B_1(t_1, t_2)$ varies inside the ellipse.
\end{enumerate}
The eigenvalues have been calculated numerically using FreeFEM++ \cite{hecht2012new} and they are given below in tabular form. Based on these values, we make various observations and state them as conjectures.
\subsection{When outer domain is a ball:} We take outer domain to be ball $B_5$ in $\mathbb{R}^{2}$ of radius $5$ centered at the origin and remove a ball of radius $1$ from its interior. On this doubly connected domain, we have provided the first two nonzero Steklov eigenvalues and Steklov Neumann eigenvalues in Table \ref{tab: ball}. Based on the observations from Table \ref{tab: ball}, we state the following conjecture: 
\begin{conjecture}
If $\Omega_{out}$ is the ball $B_{R_2}$ in $\mathbb{R}^{2}$ of radius $R_2$ centered at the origin. Then
\begin{enumerate}
    \item The first and second nonzero Steklov eigenvalues, $\sigma_1(\Omega_{(t_1, t_2)})$ and $\sigma_2(\Omega_{(t_1, t_2)})$, decrease as distance between centers of both the ball increases. In particular, among all positions of $B_1(t_1, t_2)$ inside $B_{R_2}$, eigenvalues  $\sigma_1(\Omega_{(t_1, t_2)})$ and $\sigma_2(\Omega_{(t_1, t_2)})$ attain maximum when both balls are concentric and these eigenvalues attain minimum when the inner ball touches the outer ball.

    \item  The first nonzero mixed Steklov Neumann eigenvalue $\mu_1(\Omega_{(t_1, t_2)})$ (with Steklov condition on the outer boundary and Neumann condition on the inner boundary) has multiplicity $2$ and it is decreasing with respect to the distance between centers of both the balls.
\end{enumerate}
\end{conjecture}

\begin{table}[h]
    \centering   
\begin{tabular}
{|c|c|c|c|c|c|c|c|}
\hline
 $(t_1, t_2)$ & (0.5, 0)  & (1, 0)  &  (1.5, 0) & (2, 0) & (2.5, 0) & (3, 0) & (3.5, 0) \\
\hline
 $\sigma_1(\Omega_{(t_1, t_2)})$ &  0.177575 &  0.175421  & 0.171908  & 0.167088 & 0.160911  & 0.152997  & 0.141689 \\
\hline
$\sigma_2(\Omega_{(t_1, t_2)})$ &  0.178242 &  0.178052  & 0.177698  & 0.177101   & 0.176084  & 0.174166 & 0.169468 \\
\hline
$\mu_1(\Omega_{(t_1, t_2)})$ & 0.184583 & 0.18448  &  0.184289  & 0.183969  & 0.183431   & 0.182441   & 0.180127  \\
\hline
$\mu_2(\Omega_{(t_1, t_2)})$ & 0.184583 &  0.18448  &  0.184289   & 0.183969  & 0.183431  & 0.182441   & 0.180127 \\
\hline
\end{tabular}
\caption{\small{Eigenvalues when outer domain $\Omega_{out}$ is ball of radius 5 centered at the origin.}}
\label{tab: ball}
\end{table}

\subsection{When outer domain is an ellipse:} Now we take ellipse $E_{3,8.33}$ in $\mathbb{R}^{2}$ centered at the origin with minor axis $3$ cm and major axis $8.33$ cm as the outer domain. In Table \ref{tab: ellipse1}, \ref{tab: ellipse2} and \ref{tab: ellipse3}, we have given the first and second nonzero Steklov and Mixed Steklov Neumann eigenvalues for different positions of $B_1(t_1, t_2)$ inside $E_{3,8.33}$. Based on the observations from these tables, we state the following conjecture.

\begin{conjecture}
If $\Omega_{out}$ is the ellipse $E_{r,R}$ in $\mathbb{R}^{2}$ centered at the origin with minor axis $r$ and major axis $R$. Then
\begin{enumerate}
    \item $\sigma_1(\Omega_{(t_1, t_2)})$ and $\mu_2(\Omega_{(t_1, t_2)})$ are decreasing with respect to the distance between centers of $E_{r,R}$ and $B_1(t_1, t_2)$.
    \item $\mu_1(\Omega_{(t_1, t_2)})$ is decreasing with respect to the distance between centers of $E_{r,R}$ and $B_1(t_1, t_2)$ when center of $B_1(t_1, t_2)$ varies along lines $y=0$ and $y=x$. However, it is increasing when center of $B_1(t_1, t_2)$ varies along line $x=0$.
    
\end{enumerate}
\end{conjecture}

\begin{table}[h]
    \centering
    
    \begin{tabular}
    {|c|c|c|c|c|c|}
    \hline
      $ (t_1, t_2)$ &(0.4, 0)  &(0.8, 0)   &(1.2, 0)   &(1.6, 0)   &(1.9, 0)   \\
    \hline
      $\sigma_1(\Omega_{(t_1, t_2)})$ &0.0671757 & 0.0670031 & 0.0666325 & 0.0657731 & 0.0636588  \\
    \hline
   $\sigma_2(\Omega_{(t_1, t_2)})$   & 0.200234  & 0.195403  & 0.186742 & 0.17198 & 0.1478 \\
    \hline
    $\mu_1(\Omega_{(t_1, t_2)})$   &  0.0682314 & 0.0680922 & 0.0677969 & 0.0671274 & 0.0655633  \\
    \hline
    $\mu_2(\Omega_{(t_1, t_2)})$    & 0.231544  & 0.230396  & 0.228218  & 0.22400   & 0.214635 \\
    \hline
   
    \end{tabular}
    \caption{\small{Eigenvalues when outer domain is $E_{3,8.33}$ and center of inner ball varies along $X-$axis}}
    \label{tab: ellipse1}
\end{table}

\begin{table}[h]
    \centering    
\begin{tabular}{|c|c|c|c|c|c|c|c|}
\hline
 $(t_1, t_2)$    &(0, 0.5) & (0, 1.5)& (0, 2.5) & (0, 3.5) & (0, 4.5)& (0, 5.5)& (0, 6.5)         \\
\hline
  $\sigma_1(\Omega_{(t_1, t_2)})$  & 0.0671408 &  0.0664638  & 0.0651789  & 0.0633908 & 0.0611841 & 0.0585319 & 0.0548775 \\
\hline
 $\sigma_2(\Omega_{(t_1, t_2)})$  & 0.202004 & 0.203392 & 0.204995 &  0.204736 & 0.200454  & 0.190421  &0.17005 \\
\hline
  $\mu_1(\Omega_{(t_1, t_2)})$  & 0.0682859 & 0.0683868 & 0.0685867 & 0.0688788 & 0.0692449 & 0.0696235 & 0.0697002 \\
\hline
 $\mu_2(\Omega_{(t_1, t_2)})$   & 0.231531 & 0.228691 & 0.223756  & 0.217687 & 0.211176 & 0.204242 & 0.19409    \\
\hline
\end{tabular}
\caption{\small{Eigenvalues when outer domain is $E_{3,8.33}$ and center of inner ball varies along $Y-$axis}}
\label{tab: ellipse2}
\end{table}


\begin{table}[h]
  \centering  
\begin{tabular}{|c|c|c|c|c|}
\hline
  $(t_1, t_2)$  &  (0.5, 0.5) &(1, 1)&(1.5, 1.5) & (1.9, 1.9)       \\
\hline
$\sigma_1(\Omega_{(t_1, t_2)})$ & 0.0670582 & 0.0664918&  0.0651753 & 0.0600965     \\
\hline
$\sigma_2(\Omega_{(t_1, t_2)})$ & 0.199557   & 0.192501 & 0.177751  & 0.128205       \\
\hline
  $\mu_1(\Omega_{(t_1, t_2)})$ & 0.0682192 & 0.0680132 & 0.0673899 & 0.0641569     \\
\hline
 $\mu_2(\Omega_{(t_1, t_2)})$  & 0.230959 & 0.227984  & 0.221902 & 0.198211   \\
\hline
\end{tabular}
\caption{\small {Eigenvalues when outer domain is $E_{3,8.33}$ and center of inner ball varies along line $y=x$.}}
\label{tab: ellipse3}
\end{table}

\newpage
\textbf{Acknowledgement:} S. Basak is supported by University Grants Commission, India. The corresponding author S. Verma acknowledges the project grant provided by SERB-SRG sanction order No. SRG/2022/002196.

    \bibliographystyle{plain}
\bibliography{main}

\end{document}